\documentclass[11pt]{amsart}

\usepackage{amsmath, amsthm, amssymb, amsfonts}
\usepackage{graphics, epsfig}
\usepackage{graphicx}
\usepackage{amscd}
\usepackage{subcaption} 
\usepackage{caption}
\graphicspath{{images/}}
\usepackage{comment}
\usepackage{array,booktabs,calc}

\usepackage{float}

\usepackage{pb-diagram}

\usepackage{lmodern}
\usepackage{palatino}                       % text font
\usepackage[bigdelims,vvarbb]{newpxmath}    % math font; automatically loads mathtools
\usepackage[scaled=0.95]{inconsolata}       % Sans font 
\usepackage[T1]{fontenc}                    % font encoding 

\usepackage{comment}

\usepackage[abs]{overpic}		  % overpic automatically loads graphicx

\usepackage{xcolor}
\definecolor{indigo}{rgb}{0.29, 0.0, 0.51}  % custom colors
\usepackage[colorlinks, urlcolor=indigo, linkcolor=indigo, citecolor=indigo]{hyperref}

\usepackage[hcentering, vcentering, total={5.9in, 8.2in}]{geometry}  % 1.4 vertical margin; 1.3 horizontal margin

\definecolor{dblue}{RGB}{33, 64, 154}

\theoremstyle{plain}
\newtheorem{theorem}{Theorem}

\newtheorem{lemma}[theorem]{Lemma}

\theoremstyle{definition}

\theoremstyle{remark}
\newtheorem{remark}[theorem]{Remark}

\numberwithin{theorem}{section}

\newcommand{\dfn}[1]{{\em #1}}        % definition
\newcommand{\R}{\mathbb{R}}           % the real numbers
\newcommand{\Q}{\mathbb{Q}}           % the rational numbers
\makeatletter
\newcommand*\bigcdot{\mathpalette\bigcdot@{0.6}}
\newcommand*\bigcdot@[2]{\mathbin{\vcenter{\hbox{\scalebox{#2}{$\m@th#1\bullet$}}}}}
\makeatother

\DeclareMathOperator{\Tight}{Tight}
\DeclareMathOperator\tb{tb}                               % Thurston-Bennequin
\DeclareFontFamily{U} {cmr}{}
\DeclareFontShape{U}{cmr}{m}{n}{
  <-6> cmr5
  <6-7> cmr6
  <7-8> cmr7
  <8-9> cmr8
  <9-10> cmr9
  <10-12> cmr8
  <12-> cmr9}{}
\DeclareSymbolFont{Xcmr} {U} {cmr}{m}{n}
\DeclareMathSymbol{\Phi}{\mathord}{Xcmr}{8}

\begin{document}

% title
\title{Surgeries on the trefoil and symplectic fillings} 

\author{John Etnyre}

\author{Nur Sa\={g}lam}

\address{School of Mathematics \\ Georgia Institute of Technology \\ Atlanta, GA}
\email{etnyre@math.gatech.edu}
\email{nsaglam6@gatech.edu}

%\subjclass[2020]{57K43}

% abstract
\begin{abstract}
In this note we will determine which contact structures on manifolds obtained by certain surgeries on the right handed trefoil are Stein fillable and which are not. This continues a long line of research and shows that there seems to be few underlying patterns to when a contact structure is Stein fillable or not. 
\end{abstract}

\maketitle
%\tableofcontents

%%%%%%%%%%%%%%%%%%%%%%%%%%%%%%%%%%%
\section{Introduction}
%%%%%%%%%%%%%%%%%%%%%%%%%%%%%%%%%%%

There has been a long history of trying to understand the symplectic fillability of contact $3$-manifolds and in particular the distinction between different types of fillability. In this paper we will focus on the distinction between Stein fillability and exact fillability as opposed to strong fillability, see Section~\ref{sympfill} for definitions. From the definition it is clear that a Stein fillable contact structure is exactly fillable, and an exactly fillable contact structure is strongly fillable. Though we do not address these issues here it is also true that a strongly fillable contact structure is weakly fillable and that a weakly fillable contact structure is tight. It is known that none of the reverse implications is true, \cite{Bow12, Eli96, EH02, Ghi05}. 

While there has been some work about the distinctions between the types of fillability on hyperbolic $3$-manifolds \cite{ConwayMin20, KalotiTosun17,LiLiu19}, the original examples showing the types of fillability were distinct were on Seifert fibered spaces (or their connect sums). In addition, we know quite a bit about the types of fillability of small Seifert fibered spaces. Recall these are Seifert fibered spaces with base $S^2$ and three singular fibers, and they have an important integer valued invariant denoted by $e_0$. We have a complete classification of contact structures on small Seifert fibered spaces with $e_0\not=-2, -1$, \cite{GLS06, Wu04}, and all of them are Stein fillable. In fact, any minimal symplectic filling of them is Stein since they are supported by planar open books \cite{Etnyre04b}. So much of the interesting behavior mentioned above occurs on small Seifert fibered spaces with $e_0=-1$. In fact, we can already see such results on small Seifert fibered spaces obtained by Dehn surgery on the right handed trefoil with surgery coefficient less than $5$. For example, in \cite{LiscaStipsicz04}, Lisca and Stipsicz showed that when the surgery coefficient is between $1$ and $4$, one has tight contact structures and none of them are fillable in any sense. And in \cite{Ghi05, GV16} Ghiggini and Van Horn-Morris showed that when the surgery coefficient is $1/n$, all tight contact structures are strongly fillable, and some are not exactly fillable, while others are Stein fillable. From their work and work of Min in \cite{Min22} one might expect a pattern to from as to when contact structures on manifolds obtained by $r\in(0,1)$ surgery on the right handed trefoil are exactly fillable and when they are not. This paper continues this investigation and shows that there are few patterns and understanding which contact structures are Stein fillable and which are not, is a very difficult question.

We now go into detail about the work of Ghiggini and Van Horn-Morris in \cite{GV16}. There they classified contact structures on $-\Sigma(2, 3, 6n+5)$, the manifold obtained from $1/(n+1)$ surgery on the right handed trefoil. They showed that $-\Sigma(2,3,6n+5)$ had exactly $n(n+1)/2$ tight contact structures. They can naturally be arranged in a triangle as in Figure~\ref{triangle}. 
\begin{figure}[htb]{\small
\begin{overpic}%[grid,tics=10] 
{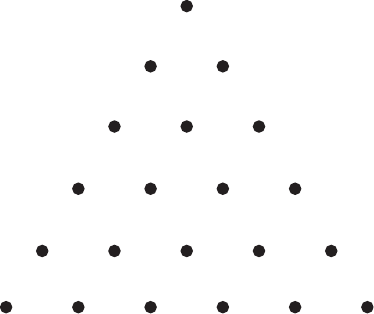}
\put(8, -5){$\xi^1_{5}$}
\put(43, -5){$\xi^1_{4}$}
\put(78, -5){$\xi^1_{3}$}
\put(113, -5){$\xi^1_{2}$}
\put(148, -5){$\xi^1_{1}$}
\put(183, -5){$\xi^1_{0}$}

\put(24, 24){$\xi^2_{4}$}
\put(61, 24){$\xi^2_{3}$}
\put(95, 24){$\xi^2_{2}$}
\put(131, 24){$\xi^2_{1}$}
\put(166, 24){$\xi^2_{0}$}

\put(43, 54){$\xi^3_{3}$}
\put(78, 54){$\xi^3_{2}$}
\put(112, 54){$\xi^3_{1}$}
\put(150, 54){$\xi^3_{0}$}

\put(61, 84){$\xi^4_{2}$}
\put(96, 84){$\xi^4_{1}$}
\put(131, 84){$\xi^4_{0}$}

\put(79, 114){$\xi^5_{1}$}
\put(113, 114){$\xi^5_{0}$}

\put(98, 144){$\xi^6_{0}$}
\end{overpic}}
\caption{Triangle representing tight contact structures on $-\Sigma(2, 3, 41)$.}
\label{triangle}
\end{figure}
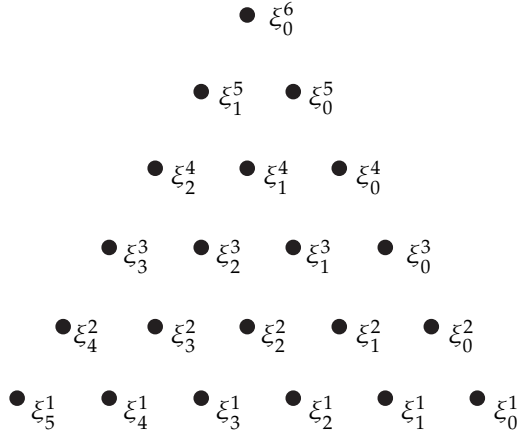

More specifically, they break into a group of $n$ structures $\{\xi^1_{n-1}, \ldots, \xi^1_1, \xi^1_0\}$, $n-1$ structures $\{\xi^2_{n-2}.\ldots, \xi^2_0\}$, all the way down to a group with only one structure $\{\xi^n_0\}$. (We are using different notation than is found in \cite{GV16} as this fits better in our discussion below.)  They, moreover, showed that the contact structures in the bottom row were all Stein fillable and the one at the top was not Stein fillable (in fact, not even exactly fillable). They conjectured that all the other rows were not exactly fillable either. It is known that all of these contact structures are strongly fillable. 

Recently, Min \cite{Min22} showed that the contact structures on the interior of the triangle are not exactly fillable; that is, there are $(n-2)(n-3)/2$, for $n\geq 3$, contact structures that are known not to be exactly fillable. In general, we do not know the fillability status of the other $2(n-2)$ contact structures on the right and left edges of the triangle. However, in upcoming work of Min, Tosun, and the first author \cite{EMT}, it will be shown that on $-\Sigma(2, 3, 23)$ all the contact structures, except the one at the top of the triangle are Stein fillable.

Generalizing the story above, we define $S^3_T(r)$ to be the result of $r$-Dehn surgery on the right handed trefoil $T$. In the upcoming work of Min, Tosun, and the first author \cite{EMT}, the contact structures on $S^3_T(r)$ will be classified (and more generally, on manifolds obtained by surgery on other torus knots). The classification for $r\in(0,1/2)$ looks significantly different than for other values of $r$. To state their result let
\[
\Phi(r)=(a_1-1)\cdots(a_n-1)
\]
where $1/r=[a_0,a_1,\ldots, a_n]$.
This will be the number of tight contact structures on a solid torus with meridional slope $r$ and dividing slope $1/n$, where $n$ is the largest integer such that $1/n$ is bigger than $r$. 
\begin{theorem}[Etnyre, Min, and Tosun]\label{class}
For a rational number $r\in [1/(n+1),1/n)$ the $3$-manifold $S^3_T(r)$ has exactly 
 \[\frac{n(n+1)}2\Phi(r)\] 
 tight contact structures up to isotopy. 
\end{theorem}
Notice that, just as for the case of $-\Sigma(2,3,6n+5)$ above, these contact structures can be arranged in a triangle, but each vertex in the triangle corresponds to $\Phi(r)$ contact structures. 
More specifically, we can denote the contact structure $\xi^k_{l,P}$ were $k$ and $l$ are as above, and $P$ is a parameter that takes on $\Phi(r)$ possible states. If $r\in [1/(n+1),1/n)$, then we say a contact structure $\xi^k_{l,P}$ is on the \dfn{interior of the triangle} if $k=2,\dots, n-1$ and $l=1,\ldots, n-k-1$, at the \dfn{top of the triangle} if $k=n$, at \dfn{the base of the triange} if $k=1$, and on the \dfn{vertical sides of the triangle} otherwise (that is if $k=2,\ldots, n-1$ and $l=0$ or $n-k$). 
This will be discussed more thoroughly in Section~\ref{sec:class}. 

One might expect that the fillability of the contact structures on $S^3_T(r)$ would follow a similar pattern to those on $-\Sigma(2,3,6n+5)=S^3_T(1/(n+1))$.  It is easy to see that the ones in the bottom row are Stein fillable and Min's result extends to show that the ones on the interior of the triangle are non-exactly fillable. So one might ask if the ones at the top of the triangle are not exactly fillable. Our first main result says that this is not the case for all $r$.

\begin{theorem}\label{main1}
If $n >3$ and 
\[
r\in \left[ \frac{2n-1}{2n^2}, \frac{2}{2n+1}\right)
\]
then $S^3_{T}(r)$ has 
\[
\frac{n(n+1)}2 \Phi(r)
\]
tight contact structures,
\begin{itemize}
\item  $(2n-1)\Phi(r)$ are Stein fillable (these are the contact structures at the base, top and half the structures at along the vertical sides of the triangle),  
\item $\frac{(n-3)(n-2)}2 \Phi(r)$ are strongly fillable, but not exact, or Stein, fillable (these are the ones in the interior of the triangle), and 
\item $(n-2)\Phi(r)$ that are strongly fillable, but we do not know if they are Stein fillable (these are half of the structures along the vertical sides of the triangel). These later contact structures are Stein fillable if and only if the contact structures at the same place on the triangle for $-\Sigma(2,3, 6n+5)$ are Stein fillable. 
 \end{itemize}

For $n=2$ or $3$ all the contact structures are Stein fillable. 
%If $n=1$ or $2$, then all the tight contact structures are Stein fillable, there are $\Phi(r)$ in the $n=1$ case and $3\Phi(r)$ in the $n=2$ case. 
%surgery on k copies of the (n,1) cable 
\end{theorem} 
%We note that the $r$ in the theorem is in $(1/(n+1),1/n)$ and so according to Theorem~\ref{class} we know the the number of tight contact structures in $\Phi(r)$ times the number of vertices in the triangle with base of size $n$. For $n>3$ we easily see the number of non-Stein fillable contact structures in the theorem is
%\[
%\left(\frac{(n-3)(n-2)}2 + (n-2) \right)\Phi(r),
%\]
%in particular, all the contact structures on the base, top, and half of those on the sides of the triangle are Stein fillable, and all those on the interior and half of those on the sides are not. 
%\begin{remark}
%We note that when $n=2$ or $3$ then all the contact structures are Stein fillable. In addition when $r=1/4$ this theorem recovers the result of the Min, Tosun, and the first author mentioned above that disproves a conjecture of Ghiggini and Van Horn-Morris
%\end{remark}
%\begin{remark}
%Notice that for any positive integer $n$ we have slopes $r_i$ approaching $1/n$ such that all the contact structures on $S^3_T(r)$ at the top of the triangle are Stein fillable, but they limit to a manifold where the top of the triangle is not exactly fillable. 
%\end{remark}

We now observe that the top of the triangle can admit both Stein fillable and non-Stein fillable contact structures. 

\begin{theorem}\label{3n1cable}
For 
\[
r\in \left[\frac{9}{25}, \frac{4}{11}\right)
\]
the manifold $S^3_T(r)$ has $3\Phi(r)$ tight contact structures of which $(2\Phi(r)+2)$ are Stein fillable and $(\Phi(r)-2)$ are strongly symplectically fillable, but not exactly symplectically fillable. 
%surgery no the (3n-1,n)-cable
\end{theorem}
We note that $r$ in this theorem is in $(1/3,1/2)$ and the proof will show that the contact structure that are not Stein fillable are all at the top of the triangle, but two at the top are Stein fillable.

\begin{theorem}\label{1329}
For 
\[
r\in \left[\frac{13}{49}, \frac{4}{15}\right)
\]
the manifold $S^3_{T}(r)$ has $6\Phi(r)$ tight contact structures of which $(5\Phi(r)+2)$ are Stein fillable and $(\Phi(r)-2)$ are strongly symplectically fillable, but not exactly symplectically fillable. 
%surgery on the (4n-1,n) cable
\end{theorem}
We note that $r$ in this theorem is in $(1/4,1/3)$ and the proof will show that all the contact structures are Stein fillable, except for $(\Phi(r)-2)$ at the top of the triangle.  

We also note that the last two theorems give a complete classification of the types of fillings that a tight contact structure can have. Apart from the result of Min, Tosun, and the first author on $-\Sigma(2,3, 23)$ mentioned above, this is the first such classification of types of fillings of $r$ surgery on the right handed trefoil when $r\in(0,1/2)$, which is the region where the most interesting things can happen. 

The main techniques used to prove the above theorems are surgeries on cables and decomposing fillings using mixed tori. More specifically, in \cite{BaldwinEtnyre13}, Baldwin and the first author introduced the notion that surgery on a Legendrian cable of a knot can sometimes be interpreted as surgery on the underlying knot too, and in  \cite{EMT}, Min, Tosun, and the first author showed that this sometimes produce Stein fillings when one could not produce such fillings using the original knot. This is how we construct most of our Stein fillings in our theorems above. To rule out Stein, or exact, fillings we use work of Christian and Menke \cite{CM22}. In that work they show that given an exact filling of a contact manifold with a special torus in it, one can ``split" the filling into fillings of other contact manifold, which in some cases do not admit exactly fillings. 

We assume the reader is familiar with basic results in contact and symplectic geometry. In Section~\ref{surgery}, we will review various results, and fix notation about Dehn surgery, surgery on cables, the classification of contact structures on solid tori, and contact surgery. In particular, we prove a results from \cite{BaldwinEtnyre13, EMT} about Legendrian surgery on cables of knots that is the key to the theorems above. In Section~\ref{sec:class}, we discuss details about, Theorem~\ref{class}, the classification of tight contact structures on surgeries on the right handed trefoil. Then in Section~\ref{sympfillsec}, we recall the notions of symplectic fillability that are of interest to us in this paper and recall the theorem of Christian and Menke \cite{CM22}. In the last two sections of the paper we prove Theorem~\ref{main1} and Theorems~\ref{3n1cable} and~\ref{1329}, respectively. 

\noindent
{\bf Acknowledgments:} The first author was partially supported by National Science Foundation grants DMS-1906414 and DMS-2203312. The second author was partially supported by T\"{U}B\.{I}TAK, the Scientific and Technological Research Council of TURKEY. Both authors were also partially supported by the Georgia Institute of Technology's Elaine M. Hubbard Distinguished Faculty Award.

%%%%%%%%%%%%%%%%%%%%%%%%%%%%%%%%%%%%%%%%%%%%%%%%%%%%%%
\section{Contact structures on solid tori and surgery on cabled knots.}\label{surgery}
%%%%%%%%%%%%%%%%%%%%%%%%%%%%%%%%%%%%%%%%%%%%%%%%%%%%%%
In this section we will discuss smooth surgeries and contact surgeries. Specifically we set up the notation for smooth surgeries on knots and discuss surgeries on cabled knots in Section~\ref{DehnSurgery}. In Section~\ref{contsurg} we discuss contact surgery and for this discussion we need to review the classification of tight contact structures on solid tori which is done in Section~\ref{classtori} after discussing the Farey graph in Section~\ref{sec:farey}

%%%%%-----------------------------------------------
\subsection{Dehn surgery and surgery on cables}\label{DehnSurgery}
%%%%%-----------------------------------------------

Let $K$ be a knot in $S^{3}$ and $\nu(K)\cong S^1\times S^2$ be the tubular neighborhood of $K$ in $S^3$.  We denote the complement of the interior of $\nu(K)$ by $S^3_K$. There is a unique curve $\mu$ on $\partial \nu(K)$ that bounds a disk in $\nu(K)$. Any other embedded curve $\lambda$ in $\partial \nu(K)$ that intersects $\mu$ transversely one time is called a longitude for $K$ and determines a framing. We will use $(\lambda, \mu)$ as a basis for the homology $H_1(\partial \nu(K))$, so every element in the homology group can be expressed as a pair of integers $(a,b)$. It is well-known that the homology class can be represented by an embedded curve on the torus if and only if $a$ and $b$ are relatively prime. For relatively prime pairs $(a,b)$ we can represent them as an element $b/a$ in the rational numbers union infinity. So an element of $\Q\cup\{\infty\}$ specifies a unique simple closed curve on the torus $\partial \nu(K)$. We call $b/a$ the \dfn{slope} of the embedded curve.

When $K$ is in $S^3$ one has the Seifert framing, that is a curve $\lambda$ on $\partial \nu(K)$ that is null-homologous in $S^3_K$. If not otherwise specified, we will always use this Seifert framing when discussing Dehn surgery. 

We denote by $S^{3}_K\left( p/q \right)$ the manifold obtained from $S^3_K$ by gluing a solid torus $S^{1}\times D^{2}$ via a diffeomorphism $\varphi: \partial \left(S^{1}\times D^{2}\right) \rightarrow  \partial S^{3}_K$ that send the curve $\{pt\}\times \partial D^2$ to the $p/q$ curve on $\partial S^3_K$. We say $S^3_K(p/q)$ is obtained from $S^3$ by \dfn{$p/q$ Dehn surgery on $K$}.

We will give a different interpretation of Dehn surgery that will be useful for our constructions. Consider $T^2\times [0,1]$. We can foliate $T^2\times \{0\}$ by curves of slope $r\in \Q\cup\{\infty\}$ (here we have fixed a basis for $H_1(T^2)$ so that we can relate simple closed curves on $T^2$ to rational numbers (union infinity) as above). Let $S_r$ be the quotient space of $T^2\times [0,1]$ with each leaf of the slope $r$ foliation on $T^2\times \{0\}$ collapsed to a point. It is clear that $S_r$ is a solid torus and the meridian in the solid torus has slope $r$ (with respect to to the given basis). Now we can think of $S^3$ as $S^3_K\cup \nu(K)$ where $\nu(K)$ is $S_\infty$ in the $\lambda, \mu$ basis. So $S^3_K(r)$ is simply $S^3_K\cup S_r$. So instead of gluing a fixed $S^1\times D^2$ to $S^3_K$ via a diffeomorphism of the boundary, we just replace $S_\infty$ with $S_r$ (where the boundaries are glued by the identity map). 

Expanding on the construction above, we can also define $S^r$ to be the quotient space of $T^2\times [0,1]$ with each leaf of the slope $r$ foliation on $T^2\times \{1\}$ collapsed to a point.  This is also clearly a solid torus. We say $S^r$ is a solid torus with \dfn{upper meridian $r$} and when necessary, for emphasis, we say $S_r$ is a solid torus with \dfn{lower meridian $r$}.

Given a knot $K$ and a tubular neighborhood $\nu(K)$, then we define the \dfn{(p,q)-cable of $K$} to be the knot, denoted $K_{p,q}$, given by a slope $q/p$ curve on $\partial \nu(K)$. Our results will crucially  rely on the following result.
\begin{lemma}\label{surgeryoncable}
The following manifolds are diffeomorphic
\[
S^3_K\left(\frac{pq \pm 1}{p^2}\right)\cong S^3_{K_{p,q}}(\lambda_T\pm 1)
\]
where the surgery coefficient $\lambda_T\pm 1$ means one more or one less than the framing $\lambda_T$ of $K_{p,q}$ coming from the torus that contains $K_{p,q}$.
\end{lemma}
This is well-known, but we give the simple proof as we will need the details in our construction. 

\begin{proof}
Recall that if $\Sigma$ is a surface in a $3$-manifold $M$ and $\gamma$ is a curve on $\Sigma$, then cutting $M$ along $\Sigma$ and re-gluing by a $\tau_\gamma^\pm$ (where $\tau_\gamma$ is a right handed Dehn twist about $\gamma$) results in the manifold $M_\gamma(\lambda_\Sigma \mp 1)$, where $\lambda_\Sigma$ is the framing of $\gamma$ induced by $\Sigma$, \cite{Rolfsen76}. Thus $S^3_{K_{p,q}}(\lambda_T\pm 1)$ is diffeomorphic to the result of cutting $S^3$ along $\partial \nu(K)$ and re-gluing the solid torus by $\tau_{K_{p.q}}^\mp$. One may easily compute that the meridian of $\nu(K)$ maps to $\pm p^2\lambda + (1\pm pq)\mu$ under this diffeomorphism. So we are removing $\nu(K)$ from $S^3$ and gluing it back in so that the meridian goes to the $(pq\pm 1)/p^2$ curve, in other words we are doing $(pq\pm 1)/p^2$ Dehn surgery on $K$. 
\end{proof}
\begin{remark}\label{CableDiffeo}
Below it will be helpful to know the diffeomorphism used in the proof above. That re-gluing map is 
\[
\begin{pmatrix}
1\mp pq& \pm p^2\\ \mp q^2& 1\pm pq
\end{pmatrix}.
\]
\end{remark}

%$$S^{3}_{K}\left( p/q \right)=\left(S^{3}-Int\left(\upsilon(K)\right)\right)\cup_{\varphi} \left(S^{1}\times D^{2}\right)$$
%
%
%This operation is called \emph{Dehn surgery} along K with slope $p/q$. The gluing diffeomorphism can be seen as defined on a $\left(p,q \right)$-curve, using a twist of $\partial \left( S^{1}\times D^{2}\right)$. Here the slope $p/q$ is parameterized by the standard meridian/longitude coordinates of K and we always assume $gcd\left(p,q \right) = 1$. $S^{1}\times D^{2}$ is build from $\partial \left( S^{1}\times D^{2}\right)$ by attaching a $3$ dimensional $2$- and $3$-handle. The gluing of $3$-handles are unique, so it is sufficient to keep track of the attaching circle $\alpha=\varphi\left({pt}\right)\times \partial D^{2}$ of the two handles.

%%%%%-----------------------------------------------
\subsection{The Farey graph}\label{sec:farey}
%%%%%-----------------------------------------------
To discuss contact structures on thicken tori and solid tori, we will need a convenient way to describe curves on a torus. We will do this with the Farey graph. The Farey graph lives in the unit disk in $\R^2$ and we consider the hyperbolic metric on the interior of the disk. To construct the Farey graph we first label the point $(0,1)$ by $0=\frac 01$ and $(0,-1)$ by $\infty=\frac 10$. We now define the Farey sum of to fractions as follows: $\frac ab \oplus \frac cd= \frac{a+c}{b+d}$. Now considering points on the boundary of the unit disk with $x$-coordinate positive, if a point is half way between two labeled points $r$ and $s$ , then label it with the Farey sum of the two points $r\oplus s$ and then connect the new point to the two old points with hyperbolic geodesics. Iterate this process until all positive rational numbers are labeled. We now do the same thing for points on the boundary of the disk with negative $x$-coordinate except now we take $\infty=\frac{-1}0$. See Figure~\ref{fareygraph}.
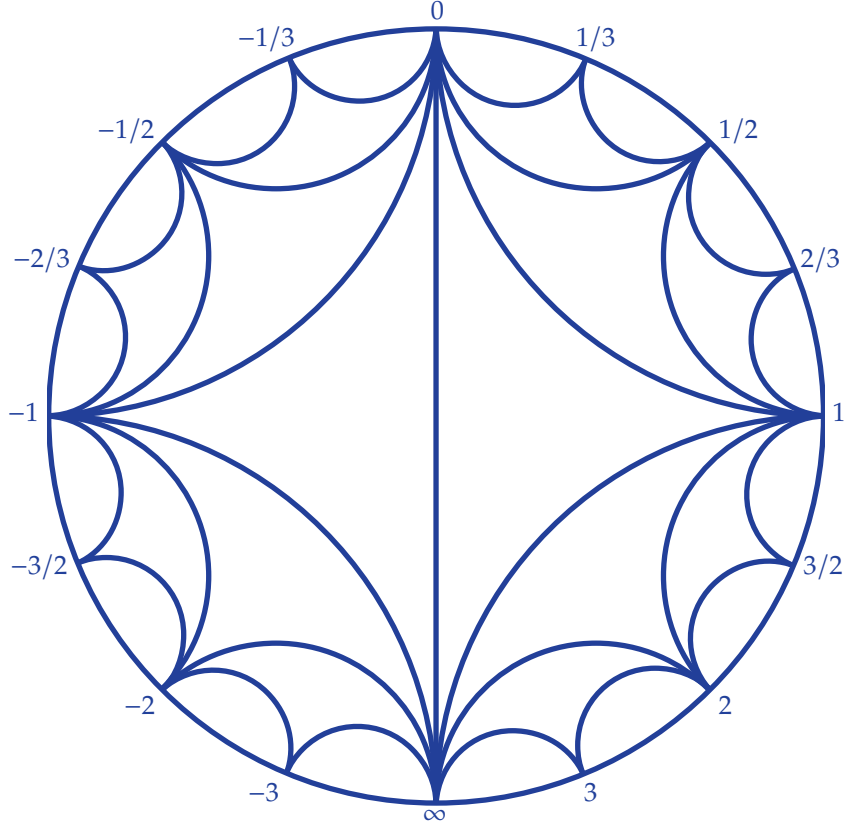
\begin{figure}[htb]{\small
\begin{overpic}%[grid,tics=10] 
{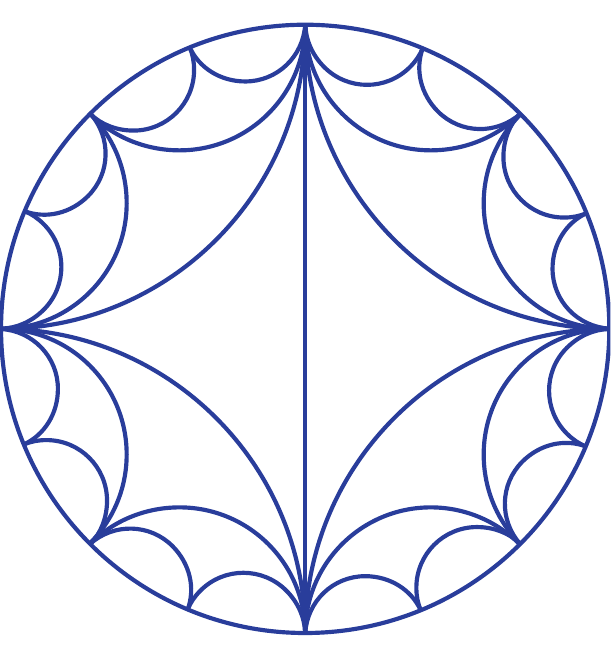}
\put(142, 2){\color{dblue}$\infty$}
\put(145, 306){\color{dblue}$0$}
\put(-15, 154){\color{dblue}$-1$}
\put(297, 154){\color{dblue}$1$}
\put(29, 43){\color{dblue}$-2$}
\put(254, 43){\color{dblue}$2$}
\put(19, 261){\color{dblue}$-1/2$}
\put(254, 261){\color{dblue}$1/2$}
\put(72, 296){\color{dblue}$-1/3$}
\put(200, 296){\color{dblue}$1/3$}
\put(-13, 212){\color{dblue}$-2/3$}
\put(285, 212){\color{dblue}$2/3$}
\put(-14, 96){\color{dblue}$-3/2$}
\put(286, 96){\color{dblue}$3/2$}
\put(76, 10){\color{dblue}$-3$}
\put(203, 10){\color{dblue}$3$}
\end{overpic}}
\caption{The Farey graph.}
\label{fareygraph}
\end{figure}

We now discuss some notation that will be used below. Given two rational numbers $r_0$ and $r_1$ we will denote by the interval $(r_0,r_1)$ all the rational numbers that are strictly clockwise of $r_0$ and strictly anti-clockwise of $r_1$. We can similarly define $[r_0,r_1]$ by removing ``strictly" from the previous definition. We will also write $r_0\oplus n r_1$ to be the result of Farey summing $r_1$, $n$ times to $r_0$.

%%%%%-----------------------------------------------
\subsection{Contact structures on thickened and solid tori}\label{classtori}
%%%%%-----------------------------------------------
Here we review the classification of tight contact structures on thickened tori and solid torus \cite{Giroux00, Hon001}, following \cite{Hon001}. We assume the reader is familiar with the basic terminology about convex surfaces as given in \cite{EtnyreHonda01b, Hon001}, though we note that our slope conventions are the reciprocal of their's, this is to better agree with topologist's slope conventions.  

When considering the manifold $T^2\times [0,1]$ we will denote $T^2\times \{i\}$ by $T_i$. Given a tight contact structure on $T^2\times [0,1]$ with convex boundary and $T_i$ having dividing slope $s_i$, we say the contact structure is \dfn{minimally twisting} if any convex torus in $T^2\times [0,1]$ that is isotopic to a boundary torus, has dividing slope clockwise of $s_0$ and anti-clockwise of $s_1$ in the Farey graph. We denote by $\Tight_m(T^2\times [0,1], s_0, s_1)$ the set of isotopy classes of tight, minimally twisting contact structures on $T^2\times [0,1]$ with convex boundary having two dividing curves each and dividing slopes $s_0$ on $T_0$ and $s_1$ on $T_1$. Honda \cite{Hon001} and Giroux \cite{Giroux00} classified the contact structures in $\Tight_m(T^2\times [0,1], s_0, s_1)$, but we will only need to consider the case when $s_0$ and $s_1$ are connected by an edge in the Farey graph. In this case $\Tight_m(T^2\times [0,1], s_0, s_1)$ has exactly two elements. These elements are called \dfn{basic slices} and are determined by a sign on the edge (this corresponds to the sign of the relative Euler class of the contact structure, but we will not need this), so we call one a positive basic slice and the other a negative basic slice. 

%Let $P$ be a minimal path in the Farey graph starting at $s_0$ and moving clockwise to $s_1$. An assignment of signs to each edge in $P$ is called a decorated path. A subset of a path $P$ of length $n$ is called a continued fraction block, if, after a change of coordinates on $T^2$, it is the path $-n, -n+1, \ldots, -1, 0$. We say two decorations of a path $P$ differ by \dfn{shuffling in continued fraction blocks} if they have the same number of $+$ signs (and hence $-$ signs) in each continued fraction block. Shuffling in continued fraction blocks gives an equivalence relation on decorated paths. 

Recall in the previous subsection we defined the solid torus $S_r$ with lower meridian $r$ and the solid torus $S^r$ with upper meridian $r$. We would now like to give the classification of tight contact structures on these tori. We let $\Tight(S_r,s)$ be the isotopy classes of tight contact structures on $S_r$ with convex boundary having two dividing curves of slope $s$. Notice that when there is an edge between $s$ and $r$ in the Farey graph, then one of the dividing curves on $\partial S_r$ is a longitude for $S_r$ and a theorem of Kanda \cite{Kanda97} says that there is a unique tight contact structure on $S_r$ with these boundary conditions. 

More generally, let $P$ be a minimal path in the Farey graph starting at $r$ and moving clockwise to $s$. An assignment of signs to each edge in $P$, except the first edge, is called a decorated path. A subset of a path $P$ of length $n$ is called a continued fraction block, if, after a change of coordinates on $T^2$, it is the path $-n, -n+1, \ldots, -1, 0$. We say two decorations of a path $P$ differ by \dfn{shuffling in continued fraction blocks} if they have the same number of $+$ signs (and hence $-$ signs) in each continued fraction block. Shuffling in continued fraction blocks gives an equivalence relation on decorated paths. 

\begin{theorem}\label{TorusClass}
The elements in $\Tight(S_r,s)$ are in one-to-one correspondence with equivalence classes of decorated paths in the Farey graph from $r$ clockwise to $s$. 
\end{theorem}

We note here that a non-minimal decorated path from $r$ to $s$ will also define a contact structure on $S_r$. We say a non-minimal path can be shortened if two adjacent edges can be removed and replaced with a single edge. We say a decorated path can be consistently shortened if the two edges being replaced have the same sign, and the shortened path will then have the new edge with the sign of the two removed edges. If one of the edges has no sign (that is the one with one vertex the meridional slope) then we can shorten and the new edge will not have a sign either. A contact structure on a solid torus given by a non-minimal path is tight if and only if it can be consistently shortened to a minimal path. We note one can reverse this process and take a minimal decorated path and split some of the edges in to several edges and then new edges will have the same sign as the original edge. This gives the same contact structure but described with a longer path. We in particular note that given any slope $t\in(r,s)$ we can lengthen the path describing a tight contact structure on $S_r$ to go through $t$. 

We have the same discussion for $\Tight(S^r,s)$ except we will be considering minimal paths from $r$ anti-clockwise to $s$ and again the first edge will not be decorated. 

An important part of the classification that will be useful for us below is that given a contact structure in $\Tight(S_r, s)$ and $t\in(r,s)$, there is a convex torus $T_t$ in $S_r$ parallel to the boundary such that $T_t$ is convex with two dividing curves of slope $t$. See \cite{Hon001}.  
%One specific count of tight contact structures is the following. 
%\begin{theorem}
%The number of tight contact structures on $S^0$ with convex boundary having two dividing curves of slope $s<-1$ is
%\[
%\left|\Tight(S^0,s)\right| = |(a_1+1)\cdots (a_{n-1}+1)a_n|.
%\]
%\end{theorem}

%%%%%-----------------------------------------------
\subsection{Contact Surgery}\label{contsurg}
%%%%%-----------------------------------------------
Given a Legendrian knot $L$ in a contact manifold $(M,\xi)$, let $N$ be a standard neighborhood of $L$. Namely, $N$ is a solid torus neighborhood of $L$ on which the contact structure is tight and the boundary of $N$ is convex with two dividing curves whose slope is given by the contact framing on $L$. In the notation of the previous section, when $L$ is null-homologous, using standard coordinates on $\partial N$, we see that the contact structure on $N$ is the unique one in $\Tight(S_\infty, \tb(L))$. We will denote $N$ by $S_\infty$ to emphasize that the meridian has slope $\infty$. (Recall our notation for solid tori from Section~\ref{DehnSurgery}.) It will be useful to note below that if $N_\pm$ is a standard neighborhood of the $\pm$ stabilization of $L$ inside of $N$, then $N\setminus N_\pm$ will be a $\pm$ basic slice. 

Contact $(r)$-surgery on $L$ is the result of removing $S_\infty$ from $M$ and replacing it with $S_{\tb(L)+r}$ on which we take any contact structure in $\Tight(S_{\tb(L)+r}, \tb(L))$. We note that this does not produce a unique contact structure on the manifold $M_L(\tb(L)+r)$, but when $r=1/n$ it does produce a unique contact structure. When $r=-1$ we call this surgery operation \dfn{Legendrian surgery on $L$}. In \cite{DingGeiges04}, it was shown that for $r<0$, contact $(r)$-surgery on $L$ can be achieved by a sequence of Legendrian surgeries on copies of stabilization of $L$ (and all of these can be done inside the standard neighborhood of $L$). Since it is known that Legendrian surgery preserves tightness \cite{Wan15} and all forms of fillability \cite{EtnyreHonda02b}, the same is true for contact $(r)$-surgery for all $r<0$. 

It will be useful to reinterpret the above result using coordinates on the neighborhood of the torus that are not ``standard". Given a Legendrian knot $L$ it will always have a standard neighborhood, but if different coordinates are chosen on the boundary of the neighborhood, then the meridional slope could be any slope $r$. So we will denote this neighborhood by $S_r$. The contact structure gives $L$ a framing and that determines the slope of the dividing curves on $\partial S_r$. Since the framing corresponds to a longitude in the torus (that is a curve that intersects the meridian one time), that slope $s$ must have an edge to $r$ in the Farey graph. (Notice we could change coordinates on the neighborhood so that the meridional slope was $\infty$ and the contact framing was an integer.)  Now Legendrian surgery on $L$ is the result of removing $S_r$ from $M$ and gluing in $S_{r\oplus s}$ on which there is a unique tight structure. More generally, for any $t\in (r,s)$ contact $(t)$-surgery on $L$ corresponds to a sequence of Legendrian surgeries as discussed above. 

We end this section by considering Legendrian surgeries on certain cables. 
\begin{lemma}\label{contactcable}
Let given a contact structure $\xi\in \Tight(S_{s/r}, t)$ on the solid torus $S_{s/r}$. Let $q/p$ be any slope in $(s/r,t)$. Recall that $\xi$ is determined by a decorated minimal path from $s/r$ clockwise to $t$ in the Farey graph.  As discussed after Theorem~\ref{TorusClass}, we can lengthen the path so that it goes through $q/p$. Denote $P_1$ the path from $s/r$ clockwise to $q/p$ and $P_2$ the path from $q/p$ clockwise to $t$. Let $L$ be a Legendrian divide on the convex torus $T$ of slope $q/p$. The result of Legendrian surgery on $L$ is the tight contact structure in $\Tight(S_{(q^2r+(1-pq)s)/((1+pq)r-p^2s)},t)$ given by the union of the paths $P_1`\cup P_2$, where $P_1`$ is obtained from $P_1$ by applying the diffeomorphism of $T^2$ given in Remark~\ref{CableDiffeo}.
\end{lemma}

\begin{proof}
Notice that we can take a neighborhood $N=T^2\times [-1,1]$ of $T$ so that $T$ is $T^2\times \{0\}$ and the contact structure is invariant in the $[-1,1]$ direction. The complement of the interior of $N$ consists of two pieces, a solid torus $S$ and a thickened torus $A$. Also note that the contact structure on $S$ is given by $P_1$. Any tight contact structure on a solid torus embeds in a tight contact structure on some lens space, and so Legendrian surgery on $L$ can be thought to take place in this tight lens space. Hence the contact structure we obtained on $S_{s/r}$ after Legendrian surgery on $L$ is tight. In particular, the contact structure on $N$ is tight after surgery on $L$. Notice that the dividing curves on both boundary components of $N$ still have slope $q/p$ (since the diffeomorphism in Remark~\ref{CableDiffeo} fixes the slope $q/p$). Moreover, the contact structure on $N$ after surgery must be minimally twisting since if not, the contact structure on the lens space would be overtwisted. Thus, we see that the contact structure on $N$ after surgery on $L$ is still an $[-1,1]$-invariant contact structure on $N$. The surgery on $L$ does not affect the contact structure on $A$ and it does not affect the contact structure on $S$ however $S$ is re-glued, as discussed in the proof of Lemma~\ref{surgeryoncable}, by the diffeomorphisms of it boundary given by the  matrix in Remark~\ref{CableDiffeo}. From this we see the tight contact structure on $S$ is given by $P_1'$ and the lemma follows. 
\end{proof}

%%%%%%%%%%%%%%%%%%%%%%%%%%%%%%%%%%%%%%%%%%%%%%%%%%%%%%
\section{Tight contact structures and surgeries on the trefoil}\label{sec:class}
%%%%%%%%%%%%%%%%%%%%%%%%%%%%%%%%%%%%%%%%%%%%%%%%%%%%%%

In the upcoming work \cite{EMT}, Min, Tosun and the first author will classify tight contact structures on manifolds obtained by surgery on torus knots in $S^3$. Here we will discuss this classification for surgeries on the right handed trefoil $K$ with surgery coefficient in $(0,1)$. To do this, we must first recall work of LaFountain, Tosun, and the first author, \cite{ELT12}. To state that result we recall that $S^3_T$ denotes the complement of a neighborhood of the trefoil $T$. 

\begin{theorem}\label{nonthickenable}
For all $n\geq 1$ there exist a unique tight contact structure $\xi_{n}$ on $S^3_{T}$ with convex boundary with two dividing curves of slope $\frac{1}{n}$ such that 
\begin{enumerate}
\item $(S^3_T,\xi_{n})$ embedded in tight $(S^{3},\xi_{std})$, and 
\item any convex torus in $\left(S^3_{T}, \xi_{n}\right)$ parallel to the boundary has dividing slope $\frac{1}{n}$.
\end{enumerate}
If $n>1$ then the solid tori $S_n^\pm$ with the contact structure given by the path in the Farey graph from $\infty$ clockwise to $1/n$ with a $\pm$ on the edge from $0$ to $1/n$ can be glued to $(S^3_T,\xi_{n})$ to get $(S^{3},\xi_{std})$ (that is there are two ways of embedding $(S^3_T,\xi_{n})$ in $(S^{3},\xi_{std})$). For $n=1$ there is a unique way of embedding $(S_T^3,\xi_1)$ whose complement is the solid torus $S_1$ with contact structure given by the path in the Farey graph from $\infty$ clockwise to $1$. 
\end{theorem}
We can now state the classification of contact structures on $S^3_T(r)$ for $r\in(0,1)$. 
\begin{theorem}
Given $r\in (0,1)$, there is a positive integer $n$ with $r\in \left[\frac1{n+1},\frac 1n\right)$ such that 
\begin{enumerate}
\item For any $k\leq n$ and contact structure $\xi\in\Tight(S_r,1/k)$ the result of gluing $(S^3_T,\xi_{k})$ to $(S_r, \xi)$ will be a tight contact structure on $S^3_T(r)$.
\item Any tight contact structure on $S^3_T(r)$ will come from such a gluing.
\item Two contact structures on $S^3_T(r)$ will be isotopic if and only if they come from gluing the same contact structures on $S^3_T$ and $S_r$. 
\end{enumerate}
\end{theorem}
We now discuss a convenient way to denote the tight contact structures on $S^3_T(r)$. Consider a contact structure constructed from gluing $(S^3_T,\xi_{k})$ to a contact structure in $\Tight(S_r,1/k)$. The latter is determined by a path in the Farey graph from $r$ clockwise to $1/k$. This path consists of a portion $P$ from $r$ to $1/n$ and then the continued fraction block $1/n, 1/(n-1), \ldots, 1/k$. There are $n-k$ edges in the latter part of the path and so there are $n-k+1$ possibilities for the, equivalence class of, ways to put signs on the path; in particular, this is determined by the number of $-$ signs in the continued fraction block. So we can denote the contact structure on $S^3_K(r)$ by $\xi^k_{l,P}$, where $l\in \{0,\ldots, n-k\}$ denotes the number of $-$ signs last continued fraction block and $P$ is a decorated path from $r$ to $1/n$. 
\begin{lemma}\label{countonT2}
If $n$ is the largest integer such that $1/n>r$, then the number of contact structures in $\Tight(S_r,1/n)$ is given by
\[
\Phi(r)=(a_1-1)\cdots(a_n-1)
\]
where $1/r=[a_0,a_1,\ldots, a_n]$.
\end{lemma}
\begin{proof}
Notice that $1/r$ is larger than $1$. We claim that the vertices in the shortest path in the Farey graph going from $1/r$ anti-clockwise to $1$ is obtained by $[a_0,a_1,\ldots, a_n]$, $[a_0, a_1, \ldots, a_n-1]$, and one continues to subtract $1$ from the last entry in the continued fraction, with the convention that $[a_0,\ldots, a_k, 1]$ is the same as $[a_0, \ldots, a_k-1]$, until one reaches $1$. Moreover, the continued fraction blocks correspond to the vertices are described by a continued fraction of a fixed length. The same algorithm was given in Section~2.3 of \cite{EtnyreRoy21} (and is well-know) for paths from a rational number less than $-1$ clockwise to $1$, and the proof of this algorithm is the same. Now from Theorem~\ref{TorusClass} we can see that the number of contact structures on a solid torus with upper meridian $1/r$ and dividing slope $1$ will be $(a_0-1)(a_1-1)\cdots(a_n-1)$ (if this is not clear this is the same argument used on Honda's classification of contact structures on solid tori \cite{Hon001}, see also \cite{EtnyreRoy21}). Also notice that the first integer in the path is $[a_0-1]$ (and in our case this is $n$). So the number of tight contact structures on a solid torus with upper meridian $1/r$ and dividing slope $a_0-1$ will be given by $(a_1-1)\cdots(a_n-1)$.

Consider the diffeomorphism of $T^2\times[0,1]$ that sends $t\in [0,1]$ to $1-t$ and exchanges the coordinates on $T^2$. This is an orientation preserving diffeomorphism and quotienting out curves on the boundary to get solid tori. This diffeomorphism will induce a diffeomorphism from $S^a$ to $S_{1/a}$ (see Section~\ref{classtori} for the notation). Above we gave the formula for the number of contact structures in $\Tight(S^{1/r}, n)$ so using this diffeomorphism the same formula gives the number of contact structures in $\Tight(S_r, 1/n)$, as claimed. 
%Now consider the diffeomorphism of the torus that exchanges the two basis elements in $H_1(T^2)$ used to identify curves on the torus with numbers. This will invert all the numbers and send paths that went anti-clockwise to ones that go clockwise. 
%
%So the path we discussed above from $1/r$ anti-clockwise to $1$ will go to a path from $r$ clockwise to $1$, and a change of basis will preserve continued fraction blocks, we see that the number of edges in the path from $r$ to $1$ is the same as the number of edges in the path from $1/r$ to $1$ and the edges are grouped into continued fraction blocks in the same way too. Lastly notice that 
\end{proof}
From the above, we see that the number of contact structures on $S^3_K(r)$ coming from $(S^3_T,\xi_{1})$ is $n\Phi(r)$ and form $(S^3_T,\xi_{2})$ is $(n-1)\Phi(r)$ and so on until we see there are $\Phi(r)$ coming from $(S^3_T,\xi_{n})$. That is there are $\frac{(n+1)n}2\Phi(r)$ contact structure, as claimed in Theorem~\ref{class}. Moreover, we can think of them as arranged in a triangle, as in Figure~\ref{triangle}, where each vertex in the triangle corresponds to $\Phi(r)$ edges. 

%We now use the results from the previous section to prove the result of the first author, Min, and Tosun that disproved the conjecture of Ghiggini and Van Horn-Morris about the Stein fillability of contact structures on $-\Sigma(2, 3, 23)$. This proof is the prototype for our proofs below establishing fillability of contact structures using surgery on cables. 
%\begin{theorem}
%Five of the six contact structures on $-\Sigma(2, 3, 23)$ are Stein fillable, while the sixth is not.
%\end{theorem}
%\begin{proof}
%The main result of \cite{} says that the contact structure 
%\end{proof}

%%%%%%%%%%%%%%%%%%%%%%%%%%%%%%%%%%%%%%%%%%%%%%%%%%%%%%
\section{Symplectic fillings and splitting of symplectic fillings}\label{sympfillsec}
%%%%%%%%%%%%%%%%%%%%%%%%%%%%%%%%%%%%%%%%%%%%%%%%%%%%%%
In this section we recall various notions of symplectic fillability and the splitting theorem of Cristian and Menke \cite{CM22} for symplectic fillings. 

%%%%%-----------------------------------------------
\subsection{Symplectic fillings}\label{sympfill}
%%%%%-----------------------------------------------
There are several different notions of symplectic fillability for contact structures. We will be interested in three of them: strong fillability, exact fillability and Stein fillability. 

A closed contact manifold $(M, \xi)$ is said to be \dfn{strongly symplectically fillable} if there is a compact symplectic manifold $(X, \omega)$ such that 
\begin{itemize}
\item $\partial X=M$ as oriented manifolds,
\item $\omega$ is exact near the boundary, 
\item a primitive $\alpha$ for $\omega$ near the boundary can be chosen so that $ker(\alpha|_ {{M}})=\xi$.
\end{itemize}
In this case, $(X, \omega)$ is called \emph{a strong symplectic filling of $(M, \xi)$}. %We note that the last two conditions can be rephrased by saying there is a vector field $v$ defined near $\partial X$ such $\mathcal{L}_v\omega=\omega$, $v$ points out of $\partial X$, and $\iota_v\omega$ is a contact form for $\xi$. Here $\mathcal{L}$ denotes Lie derivative and $\iota$ denotes contraction of a vector field into a form. 

The strong symplectic filling of $(Y,\xi)$ is said to be an \dfn{exact symplectic filling} if the primitive $\alpha$ for $\omega$ in the definition can be chosen on all of $X$. 

A \dfn{Stein manifold} is a triple  $(X,J, \phi)$ where $J$ is a complex structure on $X$ (more specifically, $J:TX\to TX$ is a bundle map induced from the complex structure on $X$), and $\phi$ is a proper exhausting function such that $\omega_\phi(v,w)=-d(d\phi\circ J)(v,w)$ is non-degenerate (and hence a symplectic form). A sub-level set of $\phi$ is called a \dfn{Stein domain}. We say a contact manifold $(M,\xi)$ is \dfn{Stein fillable} if $Y$ is the boundary of a Stein domain and $\xi=\ker(d\phi\circ J)$. 

It is clear from the definitions that a Stein fillable contact structure is exactly fillable and an exactly fillable contact structure is strongly fillable.

%%%%%-----------------------------------------------
\subsection{Splittings of symplectic fillings}
%%%%%-----------------------------------------------
A key tool in our results is a theorem of Christian and Menke that shows how to split a symplectic filling along a solid torus under certain circumstances. We need a few preliminary definitions before stating their theorem. 

Given a $4$-manifold $X$ with boundary $M$ and a properly embedded solid torus $S$ in $X$ we call $X'=X\setminus S$ the result of \dfn{splitting $X$ along $S$}. Notice that the boundary $\partial X'$ can be described by cutting $M$ along $T=\partial S$ and gluing in two solid tori to the resulting boundary components. We note that both solid tori are glued so that the meridian goes to the same curve on both torus boundary components of $M\setminus T$. We denote the resulting manifold $M'$ and say it is obtained from $M$ by \dfn{splitting along $T$}. If we have coordinates on $T$ so that curves can be indicated by a slope and the solid tori are Dehn fillings along curves of slope $s$ then we say that $M'$ is obtained from $M$ by a \dfn{splitting of slope $s$ along $T$.} %We note that $X$ may be recovered from $X'$ by attaching a round $1$-handle, that is $S^1\times D^1\times D^2$ glued along the two solid tori $S^2\times (\partial D^1)\times D^2$.

A convex torus $T$ in a symplectic manifold $(M,\xi)$ is called a \dfn{mixed torus} if there is a neighborhood $T^2\times [-1,1]$ of $T=T^2\times\{0\}$ such that the contact structure restricted to $T^2\times [-1,0]$ and to $T^2\times [0,1]$ are both basic slices and they have different signs. We call $T\times \{1\}$ and $T^2\times \{-1\}$ the \dfn{associated tori to $T$}. Let $s_i$ be the slope of the dividing curves on $T^2\times \{i\}$. Denote by $E$ the set of vertices in the Farey graph that are in the interval $[s_1,s_{-1}]$ (that is clockwise of $s_1$ and anti-clockwise of $s_{-1}$) that have an edge to $s_0$. We call $E$ the \dfn{exceptional slopes for $T$}. 

We are now ready to state Christian and Menke's splitting theorem \cite{CM22}. 

\begin{theorem}\label{cm}
If $(X,\omega)$ is a exact filing of the contact manifold $(M,\xi)$ and $T$ is a mixed convex torus. Then $X$ can be split along some solid torus $S$ with boundary $T$ so that the resulting manifold $X'$ is an exact symplectic filling of $M$ after a slope $s$ splitting along $T$ for some $s\in E$. 
\end{theorem}

We now give Min's proof about the non-exact fillability of some of the contact structures on $-\Sigma(2, 3, 6n+5)$. To do so, we first let $B$ be the torus bundle over $S^1$ obtained from $0$-surgery on the right handed trefoil and $C$ be the core of the surgery torus. It is known that $B$ has tight contact structures $\xi_n$, $n\geq 0$, where $(B,\xi_1)$ is Stein fillable and $(B,\xi_n)$, for $n>1$ is strongly but not Stein fillable, see \cite{GV16}. One may easily see that $-\Sigma(2, 3, 6n+5)$ is obtained from $-n$ surgery on $C$, see \cite{GV16}. In $\xi_1$, $C$ can be realized by a Legendrian knot $L_1$ with contact twisting $0$ (with respect to a framing coming from the bundle structure) and we notice that the complement of a standard neighborhood of $L_1$ is also the complement of the maximal Thurston-Bennequin invariant realization of the right handed trefoil in $(S^3, \xi_{std})$. So one can stabilize $L_1$, $n-1$ times to obtain a Legendrian knot with twisting $-n+1$. Then Legendrian surgery yields a contact structure on $-\Sigma(2, 3, 6n+5)$. Denote this contact structure by $\xi^1_l$ where $l$ is the number of negative stabilizations done to $L_1$. %(in the introduction, the second index indicated the rotation number of the stabilized $L_1$, but here we think the number of negative stabilizations is more useful). 
Notice that all of these contact structures are Stein fillable since so is $\xi_1$. More generally we have,
\begin{lemma}\label{firstrowStein}
Using the notation for tight contact structures on $S^3_T(r)$ established before Lemma~\ref{countonT2}, the contact structure $\xi^1_{l,P}$ is Stein fillable. 
\end{lemma}
\begin{proof}
The manifold $S^3_T(r)$ is obtained from $B$ by $-1/r$ Dehn surgery on $C$ in $B$. Since $-1/k<0$, we see that this is Legendrian surgery on $L_1$ in $(B,\xi_1)$ which is Stein fillable, hence $(S^3_T(r), \xi^1_{l,P})$ is too. 
\end{proof}

Similarly, in $(B,\xi_k)$ there is a Legendrian realization $L_k$ of $C$ with contact twisting $-k+1$. (We note that the complement of a standard neighborhood of $L_k$ is the contact structure $(S^3_K,\xi_k)$ from Theorem~\ref{nonthickenable}.)  So if we stabilize $-n+k$ times we obtain a knot with contact twisting $-n+1$. Thus, Legendrian surgery on $L_k$ yields a contact structure on $-\Sigma(2, 3, 6n+5)$ which we denote by $\xi^k_l$ where $l$ is the number of negative stabilizations done to $L_k$ before surgery.  The $\xi^k_l$ for $k\leq n$, and $l\in \{0, \ldots, n-k\}$ are all the tight contact structures on $-\Sigma(2, 3, 6n+5)$. Now Min's result \cite{Min22} says the following.
\begin{theorem}\label{minresult}
The contact structure $\xi^k_{l}$ is not exactly fillable if $l\not=0$ or $n-k$. 
\end{theorem}
We give the proof here as a warmup of our proofs below.

\begin{proof}
If $l$ is not $0$ or $n-k$, then the Legendrian knot $L_k$ discussed above was stabilized both positively and negatively to get $L'_k$ on which we performed Legendrian surgery to get $\eta_{k,l}$. Thus, the complement of a standard neighborhood of $L'_k$ contains a mixed torus $T$ with dividing curves of slope $-k+1$ and the associated tori have slopes $-k$ and $-k+2$. See the first paragraph of Section~\ref{contsurg}. Recall $K$ (and hence $L_k'$) live in $0$ surgery on the trefoil and $K$ is the core of the surgery torus. Thus surgery on $K$ can be considered surgery on the right handed trefoil (this is clear from the proof of the ``slam dunk" operation on Dehn surgery diagrams, see \cite{Rolfsen76}), and considering $T$ as a torus bounding a solid torus neighborhood of $K$, its slope is $1/(k-1)$ and its associated tori have slope $1/k$ and $1/(k-2)$. Thus we see that the only exceptional slopes for $T$ is $0$. 

Now if $(X,\omega)$ is an exact filling of $(-\Sigma(2, 3, 6n+5), \xi^k_{l})$ then Theorem~\ref{cm} says that $X$ can be split into $X'$, a filling of $S^3_K(0)$, and a lens space (since the complement of $T$ is a solid torus and $S^3_K$). Moreover, since there is an edge in the Farey graph from $1/(k-1)$ to $0$, there is a unique contact structure on this solid torus, and we see that the contact structure on $S^3_K(0)$ is $\xi_k$. Recall that any symplectic filling of a lens space must have connected boundary \cite{Etnyre04b}, and thus $X'$ is disconnected and one component of $X'$ is a strong symplectic filling of $(S^3_K(0), \xi_k)$, but this contradicts the fact that only $(S^3_K(0), \xi_1)$ is strongly fillable, thus $\xi^k_{l}$ is not exactly fillable. 
\end{proof}
Generalizing Min's work we have the following result. 
\begin{theorem}\label{generalizeMin}
Using the notation for tight contact structures on $S^3_K(r)$ established before Lemma~\ref{countonT2}, the contact structure $\xi^k_{l,P}$ is not exactly fillable if $l\not=0$ or $n-k$. 
\end{theorem}
Considering contact structures arranged in a triangle with each vertex having $\Phi(r)$ contact structures, as discussed after Lemma~\ref{countonT2}, this theorem says that any contact structure on the interior of the triangle is not exactly fillable.  
\begin{proof}
The proof is identical to that of Theorem~\ref{minresult}.
\end{proof}

%%%%%%%%%%%%%%%%%%%%%%%%%%%%%%%%%%%%%%%%%%%%%%%%%%%%%%
\section{Surgeries on the trefoil and Stein fillable contact structures}
%%%%%%%%%%%%%%%%%%%%%%%%%%%%%%%%%%%%%%%%%%%%%%%%%%%%%%
In this section we will prove Theorem~\ref{main1} that says if $n >3$ and 
\[
r\in \left[ \frac{2n-1}{2n^2}, \frac{2}{2n+1}\right)
\]
then $S^3_{T}(r)$ has 
\[
\frac{n(n+1)}2 \Phi(r)
\]
tight contact structures,
\begin{itemize}
\item  $(2n-1)\Phi(r)$ are Stein fillable (these are the contact structures at the base, top and half the structures at along the vertical sides of the triangle),  
\item $\frac{(n-3)(n-2)}2 \Phi(r)$ are strongly fillable, but not exact, or Stein, fillable (these are the ones in the interior of the triangle), and 
\item $(n-2)\Phi(r)$ that are strongly fillable, but we don't know if they are Stein fillable (these are half of the structures along the vertical sides of the triangle). These later contact structures are Stein fillable if and only if the contact structures at the same place on the triangle for $-\Sigma(2,3, 6n+5)$ are Stein fillable. 
 \end{itemize}

For $n=2$ or $3$ all the contact structures are Stein fillable.

%for $k,n >1$ and 
%\[
%r\in \left[ \frac{mn-1}{mn^2}, \frac{m}{mn+1}\right)
%\]
%then $S^3_{T}(r)$ has 
%\[
%\frac{n(n+1)}2\Phi(r)
%\] 
%tight contact structures, $(2n-1)\Phi(r)$ are Stein fillable and the reset are strongly fillable, but not exact, or Stein, fillable. 
\begin{proof}[Proof of Theorem~\ref{main1}]
Recall from Theorem~\ref{class} that there are $\frac{(n+1)n}2 \Phi(r)$ contact structures on $S^3_T(r)$. From Theorem~\ref{generalizeMin} we know that the ones on the interior of the triangle are not exactly fillable. This establishes the second bullet point.  

To establish the first bullet point we note that all of the contact structures in the base of the triangle, which contains $n\Phi(r)$ contact structures, are Stein fillable by Lemma~\ref{firstrowStein}. 

We now show that all the contact structures at the peak are Stein fillable. Recall these are $\xi^n_{0,P}$. (Here and below we are using the notation for tight contact structures on $S^3_T(r)$ established before Lemma~\ref{countonT2}.) 
%The number of contact structures on the boundary of the triangle is $(3n-3)\Phi(r)$ so we are left to show that these are Stein fillable. We know that all of the contact structures $\xi^1_{l,P}$ in the first row, consisting of $n\Phi(r)$ contact structures, are Stein fillable by Lemma~\ref{firstrowStein}. (We are using the notation for tight contact structures on $S^3_K(r)$ established before Lemma~\ref{countonT2}.)
We begin by considering the surgery $r=\frac{2n-1}{2n^2}$. From Lemma~\ref{surgeryoncable} and Remark~\ref{CableDiffeo} we know that one can achieve $\frac{2n-1}{2n^2}$ by performing surgery on $2$ copies of the $(n,1)$-cable of $T$ with surgery coefficient one less than the framing coming from the torus on which the cable sits. Now consider the solid torus $S_n^\pm$ in $(S^3,\xi_{std})$ from Theorem~\ref{nonthickenable}. We know $\partial S_\pm$ is a convex torus with two dividing curves of slope $1/n$. So the contact structure on the solid torus is described by the path from $\infty$ clockwise to $0$, and then to $1/n$. There is no sign on the first jump from $\infty$ to $0$, and all $\pm$ on the other edge. Performing Legendrian surgery on the two Legendrian divides on $\partial S_n^\pm$, we see by Lemma~\ref{contactcable} that we get $S^3_T(r)$ given by gluing $(S^3_T,\xi_k)$ to the solid torus $S_r$ with contact structure described by a path from $r$ to $1/n$ with all edges (but the first) decorated with a $\pm$. In the case that we use the $+$ sign, we see that this is the contact structure $\xi^n_{0,P_+}$ where $P_+$ is the path from $r=\frac{2n-1}{2n^2}$ to $1/n$. In Firgure~\ref{1ncablefig} we see that there are two edges in the minimal path from $r$ to $1/n$ and the second edge will have a $+$ sign on it. We have a similar discussion for the $-$ sign but get the contact structure $\xi^n_{0,P_-}$ where $P_-$ has a $-$ sign on the one edge with a sign. Thus, these two contact structures are Stein fillable. This completes the proof that all the contact structures at the top of the triangle are Stein fillable when $r=\frac{2n-1}{2n^2}$.

The only other two contact structures not accounted for by filling $(S^3_K,\xi_k)$ are $\xi^k_{0,P_-}$ and $\xi^k_{k,P_+}$. Notice both of these contact structures have a mixed torus with dividing slope $1/n$. Thus, just as in the proof of Theorem~\ref{minresult} we see that if one of these contact structures were exactly fillable then $\xi_k$ on $B$  would also be exactly fillable, but we know that these contact structures are not exactly fillable (see the paragraph after Theorem~\ref{cm} for the notation and facts about $\xi_k$). This completes the proof when $r=\frac{2n-1}{2n^2}$.

We notice that $\frac{2n-1}{2n^2}$ is given in the Farey graph by $s_{2} \oplus (n-2)s_{1}$ where $s_{1}=\frac{1}{n+1}\oplus \frac1n= \frac{2}{2n+1}$ and $s_{2}=2\frac{1}{n+1}\oplus \frac1n=\frac {3}{3n+2}$. See Figure~\ref{1ncablefig}. 
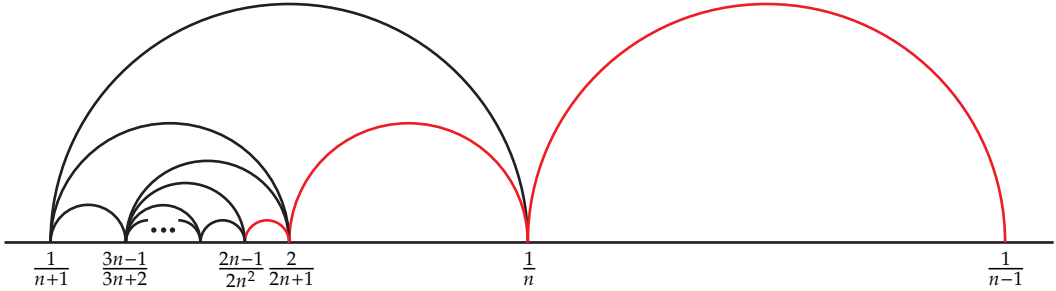
\begin{figure}[htb]{\small
\begin{overpic}%[grid,tics=10] 
{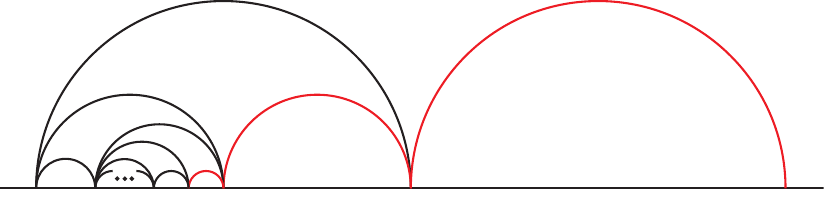}
\put(10, 2){$\frac 1{n+1}$}
\put(195, 2){$\frac 1n$}
\put(371, 2){$\frac{1}{n-1}$}
\put(99, 2){$\frac{2}{2n+1}$}
\put(80, 2){$\frac{2n-1}{2n^2}$}
\put(36, 2){$\frac{3n-1}{3n+2}$}
\end{overpic}}
\caption{Part of the path from $\frac{2n-1}{2n^2}$ to $\frac 1n$. Some of the distances are not to scale.}
\label{1ncablefig}
\end{figure}
Thus, we see there is an edge in the Farey graph from $\frac{2n-1}{2n^2}$ to $\frac 2{2n+1}$ and so in any contact structure $\xi^n_{0,P}$ on $S^3_T(\frac{2n-1}{2n^2})$ there is a solid torus $S_{(2n-1)/{2n^2}}$ with dividing slope $\frac 2{2n+1}$. This is a standard neighborhood of a Legendrian knot and any contact surgery $r\in \left(\frac{2n-1}{2n^2},\frac 2{2n+1}\right)$ can be done via Legendrian surgery on a link in the solid torus. Thus any contact structure $\xi^n_{0,P}$ on $S^3_T(r)$ is Stein fillable, see Section~\ref{contsurg}, completing the proof that all contact structures at the top of the triangle are Stein fillable. Note we have also completed the proof that all contact structures on $S^3_T(r)$ are Stein fillable when $n=2$. 

To finish the case when $n=3$, we need to see that the contact structures $\xi^2_{0,P}$ and $\xi^2_{1,P}$ in the middle row of the triangle are Stein fillable. To this end, recall we mentioned in the introduction that Min, Tosun, and the first author showed that the contact structures in the middle row for $-\Sigma(2,3,23)=S^3_T(1/4)$ were Stein fillable. They showed this by considering Legendrian surgery on the $(2,1)$ cable of $T$. In particular for the solid torus $S^\pm_2$, there are Legendrian divides of slope $1/2$ and Legendrian surgery on these gives the desired contact structures. As above, we see that there is a solid torus $S_{1/4}$ with convex boundary of slope $1/3$. Thus this is a neighborhood of a Legendrian knot and any contact surgery $r\in(1/4,1/3)$ will be Stein fillable. In particular the surgery slopes we are considering in this theorem will give Stein fillable contact structures in the second row of the triangle. 

We are left to show that, when $n>3$, on the vertical edges of the triangle, half the contact structures are Stein fillable and the other half will be if and only if so are the corresponding contact structures on $-\Sigma(2,3,6n+5)$. These contact structures are of the form $\xi^k_{0,P}$ and $\xi^k_{n-k,P}$ for $k=2, \ldots, n-1$.
% if the the last edge in $P$ has a $+$ sign and similarly for $\xi^k_{k,P}$ if the last edge in $P$ has a $-$ sign. On the other hand if the last edge in $P$ has a $-$ sign then $\xi^k_{0,P}$ is not exactly fillable by the argument above since there is a mixed torus of slope $1/n$. Similarly if the last edge of $P$ has a $+$ sign then $\xi^k_{k,P}$ is not exactly fillable. 
We begin by considering the case where $r=\frac{2n-1}{2n^2}$. As above we know that one can achieve $\frac{2n-1}{2n^2}$ by performing surgery on two copies of the $(n,1)$-cable of $T$ with surgery coefficient one less than the framing coming from the torus on which the cable sits. Now consider the solid torus $S_k^\pm$ in $(S^3,\xi_{std})$ from Theorem~\ref{nonthickenable}. Inside this torus there is a convex torus $T_\pm$ with two dividing curves of slope $1/n$. So the contact structure on the solid torus is described by the path from $\infty$ clockwise to $0$, then to $1/n$ and finally to $1/k$. There is no sign on the first jump from $\infty$ to $0$, and all $\pm$ are on the other two parts of the path. Performing Legendrian surgery on $2$ of the Legendrian divides on $T_\pm$, we see by Lemma~\ref{contactcable} that we get $S^3_T(r)$ given by gluing $(S^3_T,\xi_k)$ to the solid torus $S_r$ with contact structure described by a path from $r$ to $1/k$ with all edges (but the first) decorated with a $\pm$. In the case that we use the $+$ sign, we see that this is the contact structure $\xi^k_{0,P_+}$ where $P_+$ is the path from $r=\frac{2n-1}{2n^2}$ to $1/k$. Above we saw that there are two edges in the minimal path from $r$ to $1/n$ and the second edge will have a $+$ sign on it. We have a similar discussion for the $-$ sign but get the contact structure $\xi^k_{n-k,P_-}$ where $P_-$ has a $-$ sign on the one edge with a sign. Thus these two contact structures are Stein fillable. The only other two contact structures not accounted for by filling $(S^3_T,\xi_k)$ are $\xi^k_{0,P_-}$ and $\xi^k_{n-k,P_+}$. Notice that both of these contact structures have a mixed torus with dividing slope $1/n$. Thus, we see from Figure~\ref{1ncablefig} that there is only one exceptional slope $\{1/(n+1)\}$ for this torus and so just as in the proof of Theorem~\ref{minresult}, we see that if one of these contact structures were exactly fillable then $\xi^k_0$ and $\xi^k_{n-k}$ on $-\Sigma(2,3,6n+5)=S^3_T(1/(n+1))$ would be strongly fillable (see the beginning of the introduction for the notation for these contact structures).  Moreover, if these structures were exactly or Stein fillable they would contain a solid torus $S_{1/(n+1)}$ with dividing slope $1/n$. These are neighborhoods of Legendrian knots and hence any contact surgery in $(1/(n+1),1/n)$ would yield exact or Stein fillable contact structures on $S^3_K(r)$ for $r\in(1/(n+2),1/n)$. This completes the proof when $r=\frac{2n-1}{2n^2}$.

The case for $r\in \left(\frac{2n-1}{2n^2}, \frac{2}{2n+1}\right)$ is handled through Legendrian surgery exactly like we did when discussing the contact structures at the top of the triangle above. 
\end{proof}

%%%%%%%%%%%%%%%%%%%%%%%%%%%%%%%%%%%%%%%%%%%%%%%%%%%%%%
\section{Fillability of other surgeries on the trefoil}
%%%%%%%%%%%%%%%%%%%%%%%%%%%%%%%%%%%%%%%%%%%%%%%%%%%%%%
We are now ready to prove Theorem~\ref{3n1cable} which says that 
for 
\[
r\in \left[\frac{9}{25}, \frac{4}{11}\right)
\]
the manifold $S^3_T(r)$ has $3\Phi(r)$ tight contact structures of which $(2\Phi(r)+2)$ are Stein fillable and $(\Phi(r)-2)$ are strongly symplectically fillable, but not exactly symplectically fillable. 
%surgery no the (3n-1,n)-cable

\begin{proof}[Proof of Theorem~\ref{3n1cable}]
An $r$ as in the theorem is in $(1/3,1/2)$ and so according to Theorem~\ref{class} there are $3\Phi(r)$ contact structures and they organized in a triangle with base and hight two. From Lemma~\ref{firstrowStein} we know that the base of the triangle is Stein fillable. So we are left to see what happens for the top of the triangle. These contact structures are all of the form $\xi^2_{0,P}$ for some signed path from $r$ to $1/2$, where we are using the notation for tight contact structures on $S^3_T(r)$ established before Lemma~\ref{countonT2}. 

We begin by considering the surgery coefficient $r=\frac 9 {25}$. We can see, using Lemma~\ref{surgeryoncable}, that this surgery can be effected by performing surgery on the $(5,2)$-cable of $T$ with surgery coefficient one less than the framing determined by the cable torus. Consider the torus $S^\pm_2$ inside $(S^3,\xi_{std})$ from Theorem~\ref{nonthickenable}. Recall that the dividing curves on $\partial S^\pm_2$ have slope $1/2$. As discussed in Section~\ref{classtori}, we know that inside of $S^\pm_2$ there is a convex torus $T_\pm$ parallel to the boundary that has dividing slope $2/5$. The contact structure on $S^\pm_2$ is given by a path in the Farey graph that starts at $\infty$ goes to $0$, then $1/3$, then $2/5$ and finally $1/2$. All the edges (except the first which has no sign) have a $\pm$ sign. Now according to Lemma~\ref{contactcable} we see that performing Legendrian surgery on a Legendrian divide on $T_\pm$ will result in the solid torus $S_{9/25}$ with boundary slope $1/2$ and the contact structure is described by a path $P_\pm$ from $9/25$ to $1/2$ with all edges (except the first which has no sign) having a $\pm$ sign. That is the contact structure on $S^3_T(9/25)$ is obtained by gluing $(S^3_T, \xi_2)$ to $S_{9/25}$ with the contact structure given by $P_\pm$. Namely, when we have a $+$ sign we get $\xi^2_{0,P_+}$ and when we have a negative sign we get $\xi^2_{1,P_-}$. In particular we see that these two contact structures at the top of the triangle are Stein fillable. 

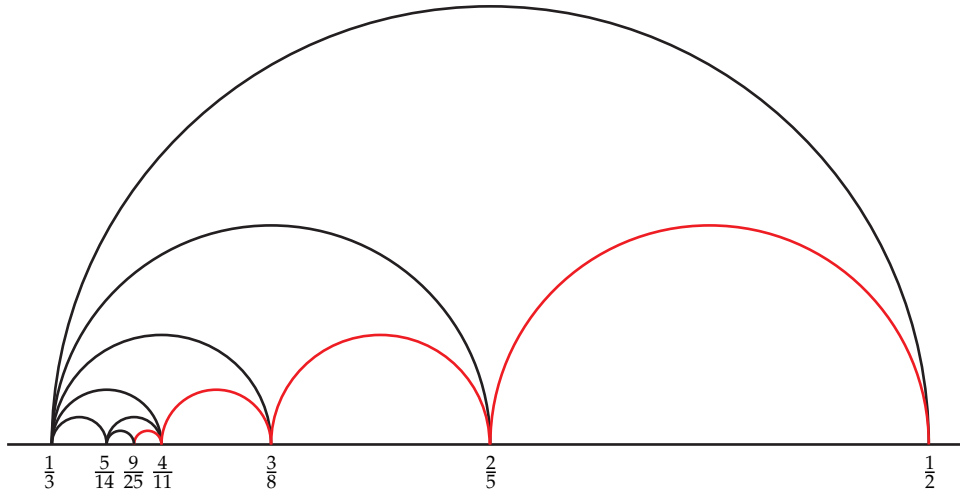
\begin{figure}[htb]{\small
\begin{overpic}%[grid,tics=10] 
{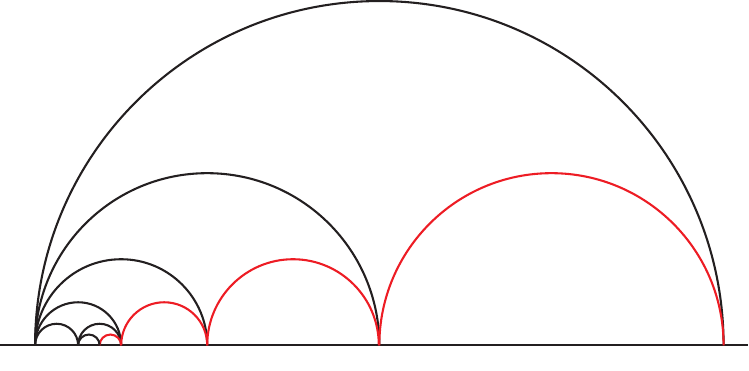}
\put(13, -1){$\frac 1{3}$}
\put(32, -1){$\frac 5{14}$}
\put(43, -1){$\frac 9{25}$}
\put(54, -1){$\frac 4{11}$}
\put(97, -1){$\frac 3{8}$}
\put(180, -1){$\frac 2{5}$}
\put(346, -1){$\frac 1{2}$}
\end{overpic}}
\caption{The path in the Farey graph from $\frac{9}{25}$ to $\frac 12$.}
\label{925}
\end{figure}
Now let $P$ be any signed path from $9/25$ to $1/2$ determining a contact structure on $S_{9/25}$ that has both signs. Figure~\ref{925} shows the path form $9/25$ to $1/2$.
Notice that the edges with a sign are all in a continued fraction block. So if both signs are present, we can assume we have a mixed torus with dividing slope $3/8$ and the exceptional slopes are $\{1/3\}$. Thus, if the contact structure $\xi^2_{0,P}$ is strongly fillable, then so is $\xi^2_0$ on $-\Sigma(2,3, 17)=S^3_T(1/3)$, but Ghiggini's result \cite{Ghi05}, see \cite{GV16}, says this is not the case. Thus, none of these contact structures are Stein fillable. This completes the theorem in the case that $r=9/25$.  

We now consider $r\in(9/25,4/11)$. Notice that in the contact structures $\xi^2_{0,P_+}$ and  $\xi^2_{1,P_-}$ on $S^3_T(9/25)$, there is a torus $S_{9/25}$ with dividing slope $4/11$. Since there is an edge between $9/25$ and $4/11$, this is the neighborhood of a Legendrian knot  and any contact surgery with slope in $(9/25, 4/11)$ can be done via Legendrian surgery on a link in the solid torus. Thus, all the contact structures $\xi^2_{0,P}$ where the last three edges in $P$ have a $+$ sign or all have a $-$ are Stein fillable. The argument at then end of the previous paragraph shows that all the other contact structures $\xi^2_{0,P}$ are not Stein or exactly fillable. 
\end{proof}

We now turn to the proof of Theorem~\ref{1329} that says for 
\[
r\in \left[\frac{13}{49}, \frac{4}{15}\right)
\]
the manifold $S^3_{T}(r)$ has $6\Phi(r)$ tight contact structures of which $(5\Phi(r)+2)$ are Stein fillable and $(\Phi(r)-2)$ are strongly symplectically fillable, but not exactly symplectically fillable. 
\begin{proof}[Proof of Theorem~\ref{1329}]
Notice that such an $r$ is in $(1/4,1/3)$ and so according to Theorem~\ref{class} there are $6\Phi(r)$ contact structures and they organized in a triangle with base and hight three.  From Lemma~\ref{firstrowStein}, we know that the base of the triangle is Stein fillable. 

We now consider the case when $r=\frac{13}{49}$. For the top of the triangle we note that the path in the Farey graph from $13/49$ to $1/3$ is shown in Figure~\ref{1349}. 
\begin{figure}[htb]{\small
\begin{overpic}%[grid,tics=10] 
{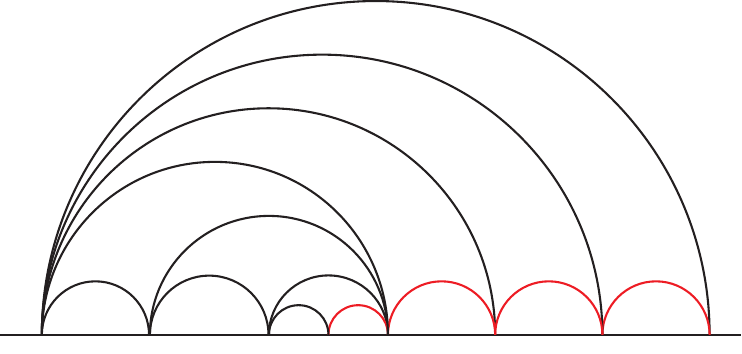}
\put(17, -3){$\frac 1{4}$}
\put(66, -3){$\frac 5{19}$}
\put(124, -3){$\frac 9{34}$}
\put(154, -3){$\frac {13}{49}$}
\put(181, -3){$\frac 4{15}$}
\put(234, -3){$\frac 3{11}$}
\put(287, -3){$\frac 2{7}$}
\put(340, -3){$\frac 1{3}$}
\end{overpic}}
\caption{The path in the Farey graph from $\frac{13}{49}$ to $\frac 13$. Some of the distances are not to scale.} 
\label{1349}
\end{figure}
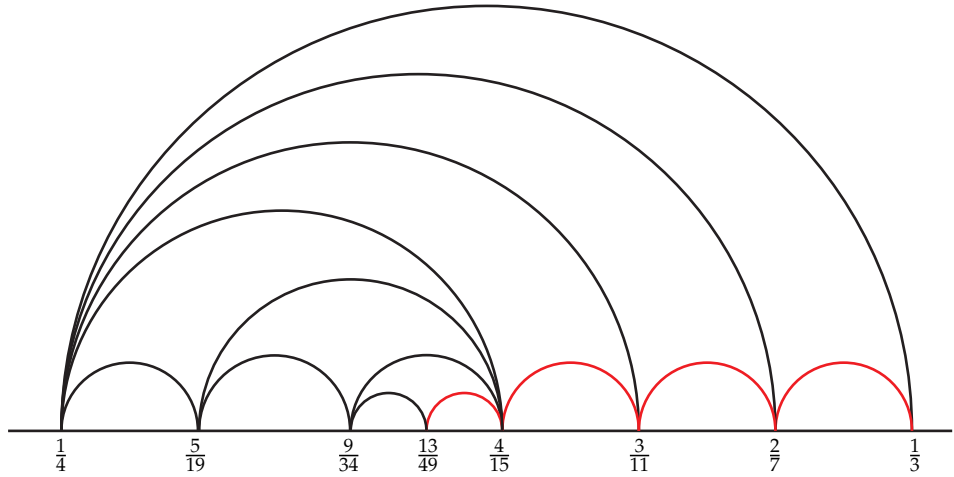
In particular there will be edges from $13/49$ to $4/15$, from $4/15$ to $3/11$, from $3/11$ to $2/7$, and from $2/7$ to $1/3$, and the last three are in a continued fraction block. Thus, noting that $13/49$ surgery on $T$ is the same as surgery on the $(7,2)$-cable of $T$ with surgery coefficient one less than the cable torus framing, we see that the argument in the proof of Theorem~\ref{3n1cable} shows that there are $2$ Stein fillable contact structures at the top of the triangle and the rest are not Stein fillable. Thus, we are left to show that the contact structures in the second row of the triangle are Stein fillable, but this follows the same argument as the one given in the sixth paragraph of the proof of Theorem~\ref{main1}. 

The case of $r\in (\frac{13}{49}, \frac{4}{15})$ can be dealt with as we did in the last paragraph of the proof of Theorem~\ref{3n1cable}
\end{proof}

%
%\bibliographystyle{plain}
%\bibliography{NurLibrary.bib} 

\end{document}